\documentclass[preprint]{elsarticle}
\usepackage{amssymb,amscd,amsmath}
\usepackage{lineno,hyperref}
\modulolinenumbers[5]

\journal{ARS Mathematica Contemporanea}
\newtheorem{thm}{Theorem}[section]
\newtheorem{prop}[thm]{Proposition}
\newtheorem{lem}[thm]{Lemma}
\newtheorem{cor}[thm]{Corollary}

\newproof{proof}[rem]{Proof}

\newcommand{\forme}[1]{}

\newcommand{\gn}[1]{\langle {#1}\rangle}
\newcommand{\Qd}{/\!\!/}
\newcommand{\Ot}{{\mathbf O}_{\mathrm \theta}}
\newcommand{\Or}{{\mathbf O}^{\mathrm \theta}}

\begin{document}

\begin{frontmatter}

\title{Association schemes with a certain type of $p$-subschemes}


\author{Abbas Wasim\corref{mycorrespondingauthor}}
\ead{wasimabbas@pusan.ac.kr}
\author{Mitsugu Hirasaka }
\address{Department of Mathematics, College of Sciences, Pusan National University,
63 Beon-gil 2, Busandaehag-ro, Geumjung-gu, Busan 609-735, Korea.}


\cortext[mycorrespondingauthor]{Corresponding author}

\fntext[1]{~\ This work was supported by BK21 Dynamic Math Center, Department of Mathematics at Pusan National University.}

\begin{abstract}
	In this article, we focus on association schemes with some properties derived
	from the orbitals of a transitive permutation group $G$ with a one-point stabilizer $H$
	satisfying $H <N_G(H)<N_G(N_G(H))\unlhd G$ and $|N_G(N_G(H))|=p^3$
	where $p$ is a  prime. 
	By a corollary of our main result we obtain some inequality
	which corresponds to the fact $|G:N_G(N_G(H))|\leq p+1$.
\end{abstract}

\end{frontmatter}   

\section{Introduction}
Let $G$ be a finite group with a subgroup $H$ which satisfies
\begin{equation}\label{eq:0}
H <N_G(H)<N_G(N_G(H))\unlhd G\mbox{ and } |N_G(N_G(H))|=p^3
\end{equation}
where $p$ is a prime.
In this article we focus on association schemes axiomatizing
some properties derived from the orbitals of the action of $G$ on $G/H$.

We shall recall some facts on coherent configurations obtained from a permutation group.
Let $G$ be a permutation group of a finite set $\Omega$.
Then $G$ acts on $\Omega\times \Omega$ by its entry-wise action, i.e.,
\[(\alpha,\beta)^x:=(\alpha^x,\beta^x) \mbox{ for $\alpha,\beta\in \Omega$ and $x\in G$.}\]
We denote the set of orbits of the action of $G$ on $\Omega\times \Omega$ by $ \mathrm{Orb}{_2}(G) $ and $\mathrm{Inv}(G) =(\Omega, \mathrm{Orb}{_2}(G))$, which satisfies the following conditions:
\begin{enumerate}
\item The diagonal relation $1_\Omega$ is a union of elements of $ \mathrm{Orb}{_2}(G) $;
\item For each $s\in$ $ \mathrm{Orb}{_2}(G) $ we have $s^\ast \in$ $ \mathrm{Orb}{_2}(G) $\\ where $s^\ast:=\{(\alpha,\beta)\mid (\beta,\alpha)\in s\}$;
\item For all $s,t,u\in$ $ \mathrm{Orb}{_2}(G) $ we have $\sigma_s\sigma_t=\sum_{u\in S}c_{st}^u\sigma_u$
for $c_{st}^u\in \mathbb{N}$ uniquely determined by $s,t,u$
where $\sigma_u$
is the adjacency matrix of $u$,
i.e., $(\sigma_u)_{\alpha,\beta}=1$ if $(\alpha,\beta)\in u $ and  $(\sigma_u)_{\alpha,\beta}=0$ if $(\alpha,\beta)\not\in u $.
\end{enumerate}
A \textit{coherent configuration} is a pair $(\Omega,S)$ of a finite set $\Omega$ 
and a partition $S$ of $\Omega\times \Omega$
which satisfies the conditions obtained from the above by replacing $ \mathrm{Orb}{_2}(G) $ by $S$.
We say that a coherent configuration $(\Omega, S)$ is \textit{schurian} if $S=$  $ \mathrm{Orb}{_2}(G) $ for some permutation group $G$ of $\Omega$, and it is \textit{homogeneous} or
an \textit{association scheme} if $1_\Omega\in S$ (see \cite{BI} and \cite{BCN} for its background).

Suppose that $G$ has a subgroup $H$ which satisfies (\ref{eq:0}).
Then $|H|=p$, $|N_G(H)|=p^2$ and for each $g\in G$ we have the following:
\begin{enumerate}
\item  $|HgH|/|H|\in \{ 1, p\}$ and $|N_G(H)g N_G(H)|/|N_G(H)|\in \{1,p\}$;
\item $|HgH|/|H|=1$ if and only if $g\in N_G(H)$;
\item $|N_G(H)gN_G(H)|/|N_G(H)|=1$ if and only if $g\in N_G(N_G(H))$;
\item $N_G(N_G(H))$ is the smallest normal subgroup of $G$ containing $H$. 
\end{enumerate}
Since $G$ acts faithfully and transitively on the set of right cosets of $H$ in $G$
by its right multiplication, it induces a
schurian association scheme $(\Omega,S)$
where $\Omega=\{Hx\mid x\in G\}$ and $S=$ $ \mathrm{Orb}_2(G) $ such that,
for each $s\in S$ we have the following:
\begin{enumerate}
\item $n_s\in \{1,p\}$ where $n_s:=c_{ss^\ast}^{1_\Omega}$;
\item $\Ot(S)$ forms a group of order $p$ where $\Ot(S):=\{s\in S\mid n_s=1\}$;
\item $\Or(S)=\{s\in S\mid ss^\ast s=s\}$ where $\Or(S)$ is the thin residue of $S$
(see Section~2, \cite{Z1} or \cite{Z2} for its definition).
\end{enumerate}

The following is our main result:
\begin{thm}\label{main1}
Let $(\Omega,S)$ be an association scheme with $\Ot(S)<\Or(S)$
such that $n_s\in \{1,p\}$ for each $s\in S$ and $n_{\Or(S)}=p^2$
where $p$ is a prime.
Then $|\Omega|\leq p^2(p+1)$.
\end{thm}
In \cite{CHK} one can find a criterion on association schemes
whose thin residue $\Or(S)$
induces the subschemes isomorphic to either
\[\mbox{$C_{p^2}$, $C_p\times C_p$ or $C_p\wr C_p$.}\]
Here we denote $(G,\mathrm{Inv}(G))$ by $G$ when $G$ acts on itself by its right multiplication and
we denote the wreath product of one scheme $(\Delta,U)$ by another scheme $(\Gamma,V)$ by $(\Delta,U)\wr (\Gamma,V)$, i.e.,
\[(\Delta,U)\wr (\Gamma,V):=(\Delta\times \Gamma, \{ 1_\Gamma\otimes u\mid u\in U\}\cup \{v\otimes U\mid v \in V\setminus \{1_\Gamma\}\})\]
where
\[1_\Gamma\otimes u:=\{((\delta_1,\gamma),(\delta_2,\gamma))\mid (\delta_1,\delta_2)\in u,\gamma\in \Gamma\}\] and
\[v\otimes U:=\{((\delta_1,\gamma_1),(\delta_2,\gamma_2))\mid \delta_1,\delta_2\in \Delta ,(\gamma_1,\gamma_2)\in v\}.\]
For the case of $\Or(S)\simeq C_{p^2}$ we can apply the main result in \cite{HZ} to conclude that $(\Omega,S)$ is schurian.
For the case of $\Or(S)\simeq C_p\times C_p$ we can say that $|\Omega|\leq p^2(p^2+p+1)$ under the assumption that $n_s=p$
for each $s\in S\setminus \Or(S)$.
For the case of $\Or(S)\simeq C_p\wr C_p$ we had no progression for the last five years.

In \cite{HM} all association schemes of degree 27 are
classified by computational enumeration, and there are three pairs of
non-isomorphic association schemes with $\Or(S)\simeq C_3\wr C_3$ which are algebraically isomorphic.
These examples had given an impression that we need some complicated combinatorial argument to enumerate $p$-schemes $(\Omega,S)$ with $\Or(S)\simeq C_p\wr C_p$
and $\{n_s\mid s\in S\setminus \Or(S)\}=\{p\}$. The following statement reduces the class of $ p $-schemes to those of degree $p^3$ where an association scheme $(\Omega,S)$
is called a $p$-scheme if $|s|$ is a power of $p$ for each $s\in S$:
\begin{cor}\label{cor:2}
For each $p$-scheme $(\Omega,S)$ with $\Or(S)\simeq C_p\wr C_p$,
if  $n_s=p$ for each $s\in S\setminus \Or(S)$, then $|\Omega|=p^3$.
\end{cor}

In the proof of Theorem~\ref{main1}
the theory of coherent configurations plays an important role through
the thin residue extension which is a way of construction of coherent configurations from
association schemes (see \cite[Thm.~2.1]{EP} or \cite{MP}) .
The following is the kernel of our paper:
\begin{thm}\label{thm:1}
For each coherent configuration $(\Omega,S)$ whose fibers are isomorphic to $C_p\wr C_p$,
if  $|s|=p^3$ for each $s\in S$ with $\sigma_s\sigma_s=0$,
then either $|\Omega|\leq p^2(p+1)$ or $ss^\ast s=s$ for each $s\in S$.
\end{thm}

In Section~2 we prepare necessary terminologies on coherent configurations. In Section~3 we prove our main results.

\section{Preliminaries}
Throughout this section, we assume that $(\Omega,S)$ is a coherent configuration (see \cite{chen2019coherent} for its background). Elements of $\Omega$ will be called \textit{points} and elements of $ S $ will be called \textit{basis relations}. Furthermore,
$|\Omega|$ and $|S|$ are called the \textit{degree} and \textit{rank} of $(\Omega,S)$, respectively.
For all $\alpha,\beta\in \Omega$ the unique element in $S$ containing $(\alpha,\beta)$ is denoted by $r(\alpha,\beta)$. For $s\in S$ and $\alpha\in \Omega$ we set
\[\alpha s:=\{\beta\in \Omega\mid (\alpha,\beta)\in s\}.\]
A subset $\Delta$ of $\Omega$ is called a \textit{fiber} of $(\Omega,S)$ if $1_\Delta\in S$.
For each $s\in S$, there exists a unique pair $(\Delta,\Gamma)$ of fibers such that $s\subseteq \Delta\times \Gamma$.
For fibers $\Delta,\Gamma$ of $(\Omega,S)$ we denote the set of $s\in S$ with 
$s\subseteq \Delta\times \Gamma$ by $S_{\Delta,\Gamma}$, and we set
$S_\Delta:=S_{\Delta,\Delta}$.
It is easily verified that $(\Delta,S_{\Delta})$ is a homogeneous coherent configuration.
Now we define the \textit{complex product} on the power set of $S$ as follows:
For all subsets $T$ and $U$ of $S$ we set
\[TU:=\{s\in S\mid c_{tu}^{s}>0\:\:\mbox{for some $t\in T$ and $u\in U$}\}\]
The singleton $\{t\}$ in a complex product is written without its parenthesis.

The following equations are frequently used without any mention:

\begin{lem}\label{lem:int}
Let $(\Omega,S)$ be a coherent configuration. Then for all $ r,s,t\in S, $ 
\begin{enumerate}
	\item if $rs\ne\emptyset$, then $n_rn_s=\sum_{t\in S}c_{rs}^tn_t$;  
	\item  $|t|c_{rs}^{t^\ast}=|r|c_{st}^{r^\ast}=|s|c_{tr}^{s^\ast}$;
	\item  $|\{t\in S\mid t\in rs\}|\leq \mathsf{gcd}(n_r,n_s)$. 
\end{enumerate}
\end{lem}
For $T\subseteq S_{\Delta,\Gamma}$ we set
\[n_T:=\sum_{t\in T}n_t.\]

Here we mention closed subsets, their subschemes and factor scheme 
according to the terminologies given in 
\cite{Z2}. Let $(\Omega,S)$ be an association scheme and $T\subseteq S$.
We say that a non-empty subset $T$ of $S$ is \textit{closed} if $TT^\ast \subseteq T$ 
where 
\[T^\ast:=\{t^\ast\mid t\in T\},\]
equivalently $\bigcup_{t\in T}t$ is 
an equivalence relation on $\Omega$ whose equivalence classes are 
\[\{\alpha T\mid \alpha \in \Omega\}\]
where $\alpha T:=\{\beta \in \Omega \mid (\alpha,\beta)\in t\:
\mbox{for some $t \in T$}\}$.
Let $T$ be a closed subset of $S$ and $\alpha\in \Omega$.
It  is well known (\cite{Z1}) that
\[(\Omega,S)_{\alpha T}:=(\alpha T, \{t\cap (\alpha T\times \alpha T)\mid t\in T\})\] 
is an association scheme,
called the \textit{subscheme} of $(\Omega,S)$ induced by $\alpha T$, and that
\[(\Omega,S)^T:=(\Omega/T, S\Qd T)\] 
is also an association scheme
where  
\[\mbox{$\Omega/T:=\{\alpha T\mid \alpha \in \Omega\}$,
$S\Qd T=\{s^T\mid s\in S\}$ and}\]
\[s^T:=\{(\alpha T,\beta T)\mid (\gamma,\delta)\in s\:\mbox{for some $(\gamma,\delta)\in \alpha T\times \beta T$}\},\]
which is called the \textit{factor scheme} of $(\Omega,S)$ over $T$.

We say that a closed subset $T$ is \textit{thin} if $n_t=1$ for each $t\in T$,
and the set $ \{s \in S \mid n_s =1\} $ denoted by $\Ot(S)$ is called the \textit{thin radical} of $S$, and the smallest closed subset $T$
such that $S\Qd T$ is thin is called the \textit{thin residue} of $S$, which is denoted by
$\Or(S)$.

\section{Proof of the main theorem}
Let $(\Omega,S)$ be a coherent configuration whose distinct fibers are $\Omega_1,\Omega_2,\ldots, \Omega_m$.
For all integers $i,j$ with $1\leq i,j\leq m$ we set
\[S_{ij}:=S_{\Omega_i,\Omega_j}\mbox{ and $S_i:=S_{ii}$.}\]

Throughout this section we assume that
$(\Omega_i, S_i)\simeq C_p\wr C_p$ for $i=1,2,\ldots, m$ where $p$ is a prime and
$C_p\wr C_p$ is the unique up to isomorphism non-thin $p$-scheme of degree $p^2$.

For $s\in S$ we say that $s$ is \textit{regular} if $ss^\ast s=\{s\}$ and
we denote by $R$ the set of regular elements in $S$ \cite{yoshikawa2016association}.

\begin{lem}\label{lem:31}
For each regular element $s\in S_{ij}$ with $n_s=p$ we have
\[\mbox{$\sigma_s\sigma_{s^\ast}=p(\sum_{t\in \Ot(S_i)}\sigma_t)$ and $\sigma_{s^\ast}\sigma_s=p(\sum_{t\in \Ot(S_j)}\sigma_t)$.}\]
In particular, $ss^\ast=\Ot(S_i)$ and $s^\ast s=\Ot(S_j)$.
\end{lem}
\begin{proof}
Notice that $\{1_{\Omega_i}\}\subsetneq ss^\ast\subset S_i$ and $ts=\{s\}$ for each $t\in ss^\ast$.
Since $\{t\in S_i\mid ts=\{s\}\}$ is a closed subset of valency at most $n_s$, it follows from $(\Omega_i,S_i)\simeq C_p\wr C_p$ that
$ss^\ast=\Ot(S_i)$, and hence for each $t\in ss^\ast$ we have $ n_t=1 $ and
\[c_{ss^\ast}^t=c_{st}^sn_{s^\ast}/n_{t^\ast}=p.\]
This implies that $\sigma_s\sigma_{s^\ast}=p(\sum_{t\in \Ot(S_i)}\sigma_t)$.
By symmetry, $\sigma_{s^\ast}\sigma_s=p(\sum_{t\in \Ot(S_j)}\sigma_t)$.
\end{proof}

\begin{lem}\label{lem:32}
For each non-regular element $s\in S_{ij}$ with $n_s=p$ we have
$\sigma_s\sigma_{s^\ast}=p\sigma_{1_{\Omega_i}}+\sum_{u\in S_i\setminus\Ot(S_i)}\sigma_u$ and
$\sigma_{s^\ast}\sigma_s=p\sigma_{1_{\Omega_j}}+\sum_{u\in S_j\setminus\Ot(S_j)}\sigma_u$.
\end{lem}
\begin{proof}
Notice that $\{t\in S_i\mid ts=\{s\}\}=\{1_{\Omega_i}\}$, otherwise, $s$ is regular or $n_s=p^2$, a contradiction.
This implies that the singletons $ts$ with $t\in \Ot(S_i)$ are distinct elements of valency $p$.
Since
\[p^2=|\Omega_j|=\sum_{s\in S_{ij}}n_s\geq \sum_{t\in \Ot(S_i)}n_{ts}=p+p+\cdots+p=p^2,\]
it follows that $\Ot(S_i)s=S_{ij}$.

We claim that $S_i\setminus \Ot(S_i)\subseteq ss^\ast$.
Let $u\in S_i\setminus \Ot(S_i)$.
Then there exists $t\in \Ot(S_i)$ such that $u\in tss^\ast$ since $u\in S_{ij}s^\ast=\Ot(S_i)ss^\ast$.
This implies that $u=t^\ast u\subseteq t^\ast (tss^\ast)=ss^\ast$.

By the claim with $p^2=n_sn_{s^\ast}=\sum_{t\in S_i}c_{ss^\ast t}n_t$ 
and $c_{ss^\ast 1_{\Omega_i}}=n_s=p$ we have the first statement,
and the second statement is obtained by symmetry.
\end{proof}

For the remainder of this section we assume that $n_s=p$ for each $s\in \bigcup_{i\ne j}S_{ij}$.
\begin{lem}\label{lem:33}
The set $\bigcup_{s\in R}s$ is an equivalence relation on $\Omega$.
\end{lem}
\begin{proof}
Since $1_{\Omega_i}\in S_i\subseteq R$ for $i=1,2\ldots,m$, $\bigcup_{s\in R}s$ is reflexive.
Since $ss^\ast s=\{s\}$ is equivalent to $s^\ast ss^\ast=\{s^\ast\}$, $\bigcup_{s\in R}s$ is symmetric.

Let $\alpha\in \Omega_i$, $\beta\in \Omega_j$ and $\gamma\in \Omega_k$ with $r(\alpha,\beta),r(\beta,\gamma)\in R$. 
Then we have
\[r(\alpha,\gamma)r(\alpha,\gamma)^\ast\subseteq r(\alpha,\beta)r(\beta,\gamma)r(\beta,\gamma)^\ast r(\alpha,\beta)^\ast.\]
If one of $r(\alpha,\beta)$, $r(\beta,\gamma)$ is thin, then $r(\alpha,\gamma)r(\alpha,\gamma)^\ast \subseteq \Ot(S_i)$ , and hence $r(\alpha,\gamma)\in R$. Now we assume that both of them are non-thin.
Since $r(\beta,\gamma)r(\beta,\gamma)^\ast=\Ot(S_j)=r(\alpha,\beta)^\ast r(\alpha,\beta)$, it follows that
\[r(\alpha,\gamma)r(\alpha,\gamma)^\ast\subseteq  r(\alpha,\beta) r(\alpha,\beta)^\ast=\Ot(S_i).\]
Applying Lemma~\ref{lem:31} and \ref{lem:32} we obtain that $r(\alpha,\gamma)$ is regular, and hence
$\bigcup_{s\in R}s$ is transitive.
\end{proof}
\begin{lem}\label{lem:34}
The set $\bigcup_{s\in N}s$ is an equivalence relation on $\Omega$
where $N:=\bigcup_{i=1}^mS_i\cup (S\setminus R)$.
\end{lem}
\begin{proof}
Since $1_{\Omega_i}\in S_i\subseteq N$ for $i=1,2\ldots,m$, $\bigcup_{s\in N}s$ is reflexive.
By Lemma~\ref{lem:33}, $\bigcup_{s\in R}$ is symmetric, so that $\bigcup_{s\in N}s$ is symmetric.

Let $\alpha\in \Omega_i$, $\beta\in \Omega_j$ and $\gamma\in \Omega_k$ with $r(\alpha,\beta),r(\beta,\gamma)\in N$.
Since $\bigcup_{i=1}^m S_i\subseteq R$, it follows from Lemma~\ref{lem:33}
that it suffices to show that
\[r(\alpha,\gamma)\in S\setminus R\]
under the assumption
that
\[\mbox{$r(\alpha,\beta),r(\beta,\gamma)\in S\setminus R$ with $i\ne k$.}\]
Suppose the contrary, i.e., $r(\alpha,\gamma)\in R$.
Then, by Lemma~\ref{lem:33}, $S_{ik}\subseteq R$.
Since
\[r(\alpha,\beta)r(\beta,\gamma)\subseteq S_{ik}\subseteq R,\]
it follows that
\[\Ot(S_i)r(\alpha,\beta)r(\beta,\gamma)=r(\alpha,\beta)r(\beta,\gamma).\]
On the other hand, we have
\[\Ot(S_i)r(\alpha,\beta)r(\beta,\gamma)=S_{ij}r(\beta,\gamma)=S_{ik}.\]
Thus,  $r(\alpha,\beta)r(\beta,\gamma)=S_{ik}$.
Since $i\ne k$, each element of $S_{ik}$ has valency $p$, and hence,
\[\sigma_{s_1}\sigma_{s_2}=\sum_{u\in S_{ik}}\sigma_u\]
where $s_1:=r(\alpha,\beta)$ and $s_2:=r(\beta,\gamma)$.
By Lemma~\ref{lem:32},
\[p^2=\gn{\sigma_{s_1}\sigma_{s_2},\sigma_{s_1}\sigma_{s_2}}=\gn{\sigma_{s_1}^\ast \sigma_{s_1},\sigma_{s_2}\sigma_{s_2}^\ast}=p^2+p(p-1),\]
a contradiction where $\gn{\:\:,\:\:}$ is the inner product defined by
\[\gn{A,B}:=1/p^2\mathrm{tr}(AB^\ast)\:\:\mbox{for all $A,B\in M_\Omega(\mathbb{C})$}.\] 
Therefore, $\bigcup_{s\in N}s$ is transitive.
\end{proof}

\begin{lem}\label{lem:35}
We have either $R=S$ or $N=S$.
\end{lem}
\begin{proof}
Suppose $R\ne S$.
Let $\alpha,\beta\in \Omega$ with $r(\alpha,\beta)\in R$.
Since $R\ne S$, there exists $\gamma\in \Omega$ with $r(\alpha,\gamma)\in N$.
Notice that $r(\beta,\gamma)\in R\cup N$.
By Lemma~\ref{lem:33}, $r(\beta,\gamma)\in N$, and hence, by Lemma~\ref{lem:34},
\[r(\alpha,\beta)\in R\cap N=\bigcup_{i=1}^mS_i.\]
Since $\alpha,\beta\in \Omega$ are arbitrarily taken, it follows that
\[\mbox{$R=\bigcup_{i=1}^mS_i$ and $N=S$.}\]
\end{proof}

\begin{lem}\label{lem:36}
Suppose that $S=N$ and $s_1\in S_{ij}$, $s_2\in S_{jk}$ and $s_3\in S_{ik}$ with distinct $i,j,k$.
Then $\sigma_{s_1}\sigma_{s_2}=\sigma_{s_3}(\sum_{t\in \Ot(S_k)}a_t \sigma_t)$ for some non-negative integers $a_t$ with $\sum_{t\in \Ot(S_k)}a_t=p$,
$\sum_{t\in \Ot(S_k)}a_t^2=2p-1$ and for each $u\in \Ot(S_k)\setminus \{1_{\Omega_k}\}$,
$\sum_{t\in \Ot(S_k)}a_ta_{tu}=p-1$.
\end{lem}
\begin{proof}
Since $s_1s_2\subseteq S_{ij}=s_3\Ot(S_k)$,
$\sigma_{s_1}\sigma_{s_2}=\sum_{t\in \Ot(S_k)}a_t\sigma_{s_3t}$
for some non-negative integers $a_t$.
Since $\sigma_{s_3t}=\sigma_{s_3}\sigma_t$ and 
\[p^2=n_{s_1}n_{s_2}=\sum_{t\in \Ot(S_k)}a_tn_{s_3t}=p\sum_{t\Ot(S_j)}a_t,\]
it remains to show the last two equalities on $a_t$ with $t\in \Ot(S_j)$.
Expanding $\sigma_{s_2}^\ast\sigma_{s_1}^\ast\sigma_{s_1}\sigma_{s_2}$ by two ways we obtain from Lemma~\ref{lem:32} that
\[(2p^2-p)\sigma_{1_{\Omega_j}}+(p^2-p)\sum_{t\in \Ot(S_j)\setminus\{1_{\Omega_j}\}}\sigma_t+(p^2-2p)\sum_{u\in S_j\setminus \Ot(S_j)}\sigma_u\]
\[
=\sum_{t\in \Ot(S_k)}a_t\sigma_t^\ast\sigma_{s_3}^\ast\sigma_{s_3}\sum_{t\in \Ot(S_k)}a_t\sigma_t.\]
Therefore,  we conclude from Lemma~\ref{lem:32} that
$p\sum_{t\in \Ot(S_k)}a_t^2=2p^2-p$ and $p\sum_{t\in\Ot(S_k)}a_{t}a_{tu}=p^2-p$ for each $u\in \Ot(S_k)$ with $u\ne 1_{\Omega_k}$.

\end{proof}

For the remainder of this section we assume that
\[S=N.\]
For $i=1,2,\ldots, m$ we take $\alpha_i\in \Omega_i$ and we define $t_i\in S_{i}$ such that
$t_1\in \Ot(S_{1})\setminus \{1_{\Omega_1}\}$, and for $i=2,3,\ldots,m$,
$t_i$ is a unique element in $\Ot(S_{i})$ with $r(\alpha_1t_1,\alpha_it_i)=r(\alpha_1,\alpha_i)$.
Then $C_p$ acts semi-regularly on $\Omega$ such that
\[ \Omega\times C_p\to \Omega, (\beta_i,t^j)\mapsto \beta_i t_i^j,\]
where $C_p=\gn{t}$ and $\beta_i$ is an arbitrary element in $\Omega_i$.

\begin{lem}\label{lem:37}
The  above action induces a semiregular automorphism of $(\Omega,S)$.
\end{lem}
\begin{proof}
Since $C_p$ acts regularly on each of geometric coset of $\Ot(S_i)$ for $i=1,2\ldots, m$,
the action is semi-regular on $\Omega$.
By the definition of $\{t_i\}$, it is straightforward to
show that  $r(\alpha_1,\alpha_i)$ is fixed by the action on $\Omega\times \Omega$. It follows that
each element of $\bigcup_{j=2}^mS_{1j}\cup S_{j1}$ is also fixed since $S_{1j}=\Ot(S_1)r(\alpha_1,\alpha_j)$.
Let $s\in S_{ij}$with $2\leq i,j$.
Notice that $r(\alpha_i,\alpha_1)r(\alpha_1,\alpha_j)$ is a proper subset of $S_{ij}$ by Lemma~\ref{lem:36}.
This implies that $s$ is obtained as the intersection of some of  $t_i^kr(\alpha_i,\alpha_1)r(\alpha_1,\alpha_j)$
with $0\leq k\leq p-1$, and hence $s$ is fixed.
\end{proof}
For each $i=1,2,\ldots, m$ we take $\{\alpha_{ik} \in \Omega_i \mid  k=1,2,\ldots, p \}$
{to be a complete set of representatives with respect to the equivalence relation  $\bigcup_{t\in\Ot(S_i)}t$ on $\Omega_i$.}

\begin{lem}\label{lem:38}
For each $s\in S_{ij}$ with $i\ne j$ and all $k,l=1,2,\ldots, p$
there exists a unique $h(s)_{kl}\in \mathbb{Z}_p$ such that
$r(\alpha_{ik},\alpha_{jl}t^{h(s)_{kl}})=s$.
Moreover, if $s_1\in S_{ij}$ and $t^a\in \Ot(S_k)$ with $s_1=st^a$,
then $h(s_1)_{kl}=h(s)_{kl}+a$ for all $k,l=1,2,\ldots, m$.
\end{lem}
\begin{proof}
Since $\Ot(S_j)$ acts regularly on $S_{ij}$ by its right multiplication,
the first statement holds.
For the second statement, here,\[s_1 = s t^a \implies r(\alpha_{ik},\alpha_{jl}t^{h(s_1)_{kl}}) =  r(\alpha_{ik},\alpha_{jl}t^{h(s)_{kl}})t^a \]
\[ \hspace*{1.5cm}	\implies r(\alpha_{ik},\alpha_{jl}t^{h(s_1)_{kl}}) = r(\alpha_{ik},\alpha_{jl}t^{h(s)_{kl}+a}) \]
This implies $ h(s_1)_{kl}=h(s)_{kl}+a $.
\end{proof}

\begin{lem}\label{lem:39}
For each $s\in S_{ij}$ with $i\ne j$ and all $k,l=1,2,\ldots p$
we have
\[s\cap(\alpha_{ik}\Ot(S_i)\times \alpha_{jl}\Ot(S_j))=\{(\alpha_{ik} t_i^a,\alpha_{jl}t_j^b)\mid
b-a=h(s)_{kl}\}.\]
\end{lem}
\begin{proof}
Notice that
\[r(\alpha_{ik}t_i^a,\alpha_{jl}t_j^b)=(t_i^a)^\ast r(\alpha_{ik},\alpha_{jl})t_j^b=r(\alpha_{ik},\alpha_{jl})t_j^{b-a}.\]
Since $r(\alpha_{ik},\alpha_{jl}t^{h(s)_{kl}})=s$ by Lemma~\ref{lem:38},
it follows that $r(\alpha_{ik}t_i^a,\alpha_{jl}t_j^b)=s$ if and only if $b-a=h(s)_{kl}$.
\end{proof}

\begin{prop}\label{prop:31}
For each $s\in S_{ij}$ with $i\ne j$ the matrix $(h(s)_{kl})\in M_{p\times p}(\mathbb{Z}_p)$
is such that, for all distinct $k_1,k_2\in \{1,2,\ldots, p\}$,
\[\{h(s)_{k_1,l}-h(s)_{k_2,l}\mid l=1,2,\ldots, p\}=\mathbb{Z}_p.\]
In other word the matrix is a generalized Hadamard matrix of degree $p$ over $\mathbb{Z}_p$,
equivalently, the matrix $(\xi^{h(s)_{kl}})\in M_{p\times p}(\mathbb{C})$ is a complex Hadamard matrix of Butson type $(p,p)$ where $\xi$ is a primitive $p$-th root of unity.
\end{prop}
\begin{proof}
Notice that, for all distinct $k,l$, by Lemma~\ref{lem:39},
\[\{\gamma\in \Omega\mid r(\alpha_{ik}t_i^a,\gamma)=r(\alpha_{il}t_i^b,\gamma)=s\}\]
equals
\[\bigcup_{r=1}^p\{\alpha_{jr}t_j^c\mid c-a=h(s)_{kr}, c-b=h(s)_{lr}\}.\]
Since the upper one is a singleton by Lemma~\ref{lem:32},
there exists a unique $r\in \{1,2,\ldots, p\}$ such that $b-a=h(s)_{kr}-h(s)_{lr}$.
Since $a$ and $b$ are arbitrarily taken, the first statement holds.

The second statement
holds since $\sum_{i=0}^{p-1}x^i$ is the minimal polynomial of $\xi$ over $\mathbb{Q}$.
\end{proof}

We shall write the matrix $(\xi^{h(s)_{kl}})$ as $H(s)$. For $s\in S_{ij}$ with $i\ne j$,
the restriction of $\sigma_s$ to $\Omega_i\times \Omega_j$ can be viewed as 
a $(p\times p)$-matrix whose $(k,l)$-entry is the matrix $P_i^{h(s)_{kl}}$ where $P_i$ is the permutation matrix corresponding to the mapping $\beta_i\mapsto \beta_it_i$. We may assume that $P_i=P_j$, say $P$,  for all $i,j=1,2,\ldots, m$ by Lemma~\ref{lem:37}.
Notice that $H(s)$ is obtained from $(P^{h(s)_{kl}})$ by sending $P^{h(s)_{kl}}$ to $\xi^{h(s)_{kl}}$.
\begin{prop}\label{prop:32}
For all $s_1\in S_{ij}$, $s_2\in S_{jk}$ and $s_3\in S_{ik}$ with distinct $i,j,k$
we have $H(s_1)H(s_2)=\alpha H(s_3)$ for some $\alpha\in \mathbb{C}$ with $|\alpha|=\sqrt{p}$.
\end{prop}
\begin{proof}
By Lemma~\ref{lem:36},
$H(s_1)H(s_2)=H(s_3)(\sum_{i=0}^{p-1}a_i\xi^i)$ for some $a_i\in \mathbb{Z}$
where $a_i=c_{t_k^i}$. Thus, it suffices to show that
$|(\sum_{i=0}^{p-1}a_i\xi^i)|^2=p$.
By Lemma~\ref{lem:36},
the left hand side equals
\[\sum_{i=0}^{p-1}\sum_{j=0}^{p-1}a_ia_j\xi^{i-j}=
\sum_{i=0}^{p-1}a_i^2+\sum_{i=1}^{p-1}\sum_{j=0}^{p-1}a_ja_{i+j}\xi^i
=(2p-1)+(p-1)(-1)=p.\]
\end{proof}

\begin{cor}\label{cor:31}
Let $s_i:=r(\alpha_1,\alpha_i)$ for $i=2,3,\ldots,m$ and $\mathbf{B}_i$ denote the basis for $ \mathbb{C}^p $ consisting of
the rows of $H(s_i)$ $i=2,3,\ldots,m$, and $\mathbf{B}_1$ be the standard basis.
Then $\{\mathbf{B}_1,\mathbf{B}_2,\ldots, \mathbf{B}_m\}$ is a mutually unbiased bases for $\mathbb{C}^p$,
and $m\leq p+1$.
\end{cor}
\begin{proof}
The first statement is an immediate consequence of Proposition~\ref{prop:31},
and the second statement follows from a well-known fact that the number of mutually
unbiased bases for $\mathbb{C}^n$ is at most $n+1$ (see \cite{BBRF}).
\end{proof}

\begin{flushleft}
(Proof of Theorem~\ref{thm:1})
\end{flushleft}
\begin{proof}
Suppose that $R\ne S$.
Then $N=S$ and the theorem follows from Corollary~\ref{cor:31}.
\end{proof}

\begin{flushleft}
(Proof of Theorem~\ref{main1})
\end{flushleft}
\begin{proof}
Since $n_{\Or(S)}=p^2$ and $\Ot(S)<\Or(S)$, it follows from \cite[Thm.~2.1]{EP} (or see  \cite{MP}) that the thin residue extension of $(\Omega,S)$
is a coherent configuration with all fibers isomorphic to $C_p\wr C_p$ such that
each basic relation out of the fibers has valency $p$.

We claim that $S=N$. Indeed, otherwise, $S=R$, which implies that $\gn{ss^\ast\mid s\in S}$ has valency $p$.
Since $\Or(S)=\gn{ss^\ast\mid s\in S}$ (see \cite{Z1}), it contradicts that $\Or(S)$ has valency $p^2$.

By the claim, $S=N$. Since the number of fibers of the thin residue extension of $(\Omega,S)$ equals
$|\Omega/\Or(S)|$, the theorem follows from Theorem~\ref{thm:1}.
\end{proof}

\begin{flushleft}
(Proof of Corollary~\ref{cor:2})
\end{flushleft}
\begin{proof}
Since $(\Omega,S)$ is a $p$-scheme and $\Or(S)\simeq C_p\times C_p$,
$|\Omega|$ is a power of $p$ greater than $p^2$.
By Theorem~\ref{main1}, $|\Omega|\leq (p+1)p^2$, and hence, $|\Omega|=p^3$.
\end{proof}


\bibliographystyle{amcjoucc}
\bibliography{mybibfile}

\begin{thebibliography}{10}
\expandafter\ifx\csname url\endcsname\relax
  \def\url#1{\texttt{#1}}\fi
\expandafter\ifx\csname urlprefix\endcsname\relax\def\urlprefix{URL }\fi
\expandafter\ifx\csname href\endcsname\relax
  \def\href#1#2{#2} \def\path#1{#1}\fi

\bibitem{BI}
E.~Bannai, T.~Ito,
  \href{https://books.google.co.kr/books?id=bgDvAAAAMAAJ}{Algebraic
  Combinatorics I: Association Schemes}, Benjamin/Cummings Pub. Co., 1984.
\newline\urlprefix\url{https://books.google.co.kr/books?id=bgDvAAAAMAAJ}

\bibitem{BCN}
A.~Brouwer, A.~Cohen, A.~Neumaier, Distance-regular graphs, Vol.~18,
  Springer-Verlag, Berlin, 1989.
\newblock \href {https://doi.org/10.1007/978-3-642-74341-2}
  {\path{doi:10.1007/978-3-642-74341-2}}.

\bibitem{Z1}
P.-H. Zieschang, An algebraic approach to association schemes, Lecture Notes in
  Math. 1628 (1996).
\newblock \href {https://doi.org/10.1007/BFb0097032}
  {\path{doi:10.1007/BFb0097032}}.

\bibitem{Z2}
P.-H. Zieschang, Theory of Association Schemes, Springer Science \& Business
  Media, 2005.
\newblock \href {https://doi.org/10.1007/3-540-30593-9}
  {\path{doi:10.1007/3-540-30593-9}}.

\bibitem{CHK}
J.~R. Cho, M.~Hirasaka, K.~Kim, On $p$-schemes of order $p^3$, Journal of
  Algebra 369 (2012) 369--380.
\newblock \href {https://doi.org/10.1016/j.jalgebra.2012.06.026}
  {\path{doi:10.1016/j.jalgebra.2012.06.026}}.

\bibitem{HZ}
M.~Hirasaka, P.-H. Zieschang, Sufficient conditions for a scheme to originate
  from a group, Journal of Combinatorial Theory, Series A 104~(1) (2003)
  17--27.
\newblock \href {https://doi.org/10.1016/S0097-3165(03)00104-3}
  {\path{doi:10.1016/S0097-3165(03)00104-3}}.

\bibitem{HM}
A.~Hanaki, I.~Miyamoto,
  \href{http://math.shinshu-u.ac.jp/~hanaki/as/}{Classification of association
  schemes of small order, online catalogue}.
\newline\urlprefix\url{http://math.shinshu-u.ac.jp/~hanaki/as/}

\bibitem{EP}
S.~A. Evdokimov, I.~N. Ponomarenko, Schemes of relations of the finite
  projective plane, and their extensions, Algebra i Analiz 21~(1) (2009)
  90--132, [translation in \textit{St. Petersburg Math. J.} 21 (2010), no. 1,
  65–-93].
\newblock \href {https://doi.org/10.1090/S1061-0022-09-01086-3}
  {\path{doi:10.1090/S1061-0022-09-01086-3}}.

\bibitem{MP}
M.~Muzychuk, I.~Ponomarenko, On quasi-thin association schemes, Journal of
  Algebra 351~(1) (2012) 467--489.
\newblock \href {https://doi.org/10.1016/j.jalgebra.2011.11.012}
  {\path{doi:10.1016/j.jalgebra.2011.11.012}}.

\bibitem{chen2019coherent}
G.~Chen, I.~Ponomarenko, \href{http://www.pdmi.ras.ru/~inp/}{Coherent
  configurations}, Central China Normal University Press, Wuhan (2019).
\newline\urlprefix\url{http://www.pdmi.ras.ru/~inp/}

\bibitem{yoshikawa2016association}
M.~Yoshikawa, On association schemes of finite exponent, European Journal of
  Combinatorics 51 (2016) 433--442.
\newblock \href {https://doi.org/10.1016/j.ejc.2015.07.019}
  {\path{doi:10.1016/j.ejc.2015.07.019}}.

\bibitem{BBRF}
S.~Bandyopadhyay, P.~Oscar~Boykin, C.~Vwani, V.~Farrokh, A new proof for the
  existence of mutually unbiased bases, Algorithmica 34~(4) (2002) 512--528.
\newblock \href {https://doi.org/10.1007/s00453-002-0980-7}
  {\path{doi:10.1007/s00453-002-0980-7}}.

\end{thebibliography}

%
%
%
%

\end{document}